\documentclass[a4paper,12pt]{article}
\usepackage{cmap}
\usepackage[T2A]{fontenc}
\usepackage[utf8]{inputenc} % [cp866][cp1251]
\usepackage[english]{babel}
\usepackage{amsmath}
\usepackage{amssymb}
\usepackage{graphicx}
\usepackage{ifthen}
\usepackage{multirow}
\usepackage{hyperref}
\makeatletter
\def\@seccntformat#1{\csname the#1\endcsname.\ } % the column after a section number
\def\@biblabel#1{#1.} % number style in the bibliography list
\makeatother

\date{}

\newenvironment{proof}[1][\hspace{-1.0ex}]%
{\par\addvspace{1mm}{\it Proof\hspace{1.0ex}{#1}.} }%
{\quad$\blacktriangle$\par\addvspace{1mm}}

\newif\ifNoRemark
\def\addtheorem#1#2#3#4{
\ifthenelse{\equal{#2}{}}{}%
{\ifthenelse{\expandafter\isundefined\csname the#2\endcsname}{\newcounter{#2}}{}}
\newenvironment{#1}[1][\global\NoRemarktrue]% No Remark by default
{\par\addvspace{2mm plus 0.5mm minus 0.2mm}\noindent % new paragraph without indent
{\bf #3}\ifthenelse{\equal{#2}{}}{}%
{\refstepcounter{#2}{\bf ~\csname the#2\endcsname}}%
{\bf \vphantom{##1}\ifNoRemark.\ \else\ (##1).\fi}\begingroup #4}%
{\endgroup\par\addvspace{1mm plus 0.5mm minus 0.2mm}\global\NoRemarkfalse}
\expandafter\newcommand\csname b#1\endcsname{\begin{#1}}
\expandafter\newcommand\csname e#1\endcsname{\end{#1}}
}

\addtheorem{theorem}{thrm}{Theorem}{\sl}
\addtheorem{lemma}{lmm}{Lemma}{\sl}
\addtheorem{pro}{prpstn}{Proposition}{\sl}
\addtheorem{corollary}{crllr}{Corollary}{\sl}
\addtheorem{remark}{rmrk}{Remark}{}
\addtheorem{definition}{dfnt}{Definition}{}

\title{On the OA(1536,13,2,7) and related orthogonal arrays%
\thanks{This is the final accepted version of the paper published 
in Discrete Mathematics 343(2) (2020),
paper~111659, 1--11. 
\url{https://doi.org/10.1016/j.disc.2019.111659}
distributed under
the Creative Commons CC-BY-NC-ND license
\copyright\,2019 Elsevier B.V.}%%
\thanks{This work was funded by the Russian Science Foundation under grant 18-11-00136}%
}
\author{Denis S. Krotov%
\thanks{Sobolev Inetitute of Mathematics, pr. Akademika Koptyuga 4, Novosibirsk 630090, Russia.
E-mail: \texttt{krotov@math.nsc.ru}}
}

\def\wt{\mathrm{wt}}
\def\Qq#1{Q_{#1}}

\begin{document}
\thispagestyle{empty}
\maketitle

\begin{abstract}
With a computer-aided approach based on the connection with equitable partitions, we establish the uniqueness of the orthogonal array OA$(1536,13,2,7)$, constructed in [D.G.Fon-Der-Flaass. Perfect $2$-Colorings of a Hypercube, Sib. Math. J. 48 (2007), 740--745] as an equitable partition of the $13$-cube with quotient matrix $[[0,13],[3,10]]$. By shortening the OA$(1536,13,2,7)$, we obtain $3$ inequivalent orthogonal arrays OA$(768,12,2,6)$, which is a complete classification for these parameters too.
 
After our computing, the first parameters of unclassified binary orthogonal arrays OA$(N,n,2,t)$ attending the Friedman bound $N\ge 2^n(1-n/2(t+1))$ are OA$(2048,14,2,7)$. Such array can be obtained by puncturing any binary $1$-perfect code of length $15$. We construct orthogonal arrays with these and similar parameters OA$(N=2^{n-m+1},n=2^m-2,2,t=2^{m-1}-1)$, $m\ge 4$, that are not punctured $1$-perfect codes.
 
Additionally, we prove that any orthogonal array OA$(N,n,2,t)$ with even $t$ attending the bound $N \ge 2^n(1-(n+1)/2(t+2))$ induces an equitable $3$-partition of the $n$-cube.
 
 
\end{abstract}

\section{Introduction}
Orthogonal arrays are combinatorial structures interesting from both theoretical
and practical points of view. In different applications
like design of experiments or software testing, orthogonal arrays 
are important as a good approximation of the Hamming space.
The classification of orthogonal arrays with given parameters is a problem 
that attracts attention of many researchers, see the recent works \cite{BMS:2017:few}, \cite{BulRy:2018}, and the bibliography there.

An orthogonal array OA$(N,n,q,t)$ is an $N$ by $n$ array $A$ with entries 
from $\{0,\ldots,q-1\}$ such that, within any $t$ columns of $A$, 
every ordered $t$-tuple of symbols from $\{0,\ldots,q-1\}$ occurs in exactly 
$\lambda=N/q^t$ rows of $A$.
In this paper, we characterize the orthogonal arrays 
with parameters OA$(1536,13,2,7)$ and OA$(768,12,2,6)$, which are related to each other,
lie on the Bierbrauer--Friedman \cite{Bierbrauer:95,Friedman:92} 
and  Bierbrauer--Gopalakrishnan--Stinson \cite{BGS:96} bounds, 
respectively,
and are also connected with equitable partitions 
of the $13$-cube and the $12$-cube.
These parameters were listed as open questions in Table~12.1 
of the monograph~\cite{HSS:OA}
($k=13$, $t=7$ and $k=12$, $t=6$).
An orthogonal array OA$(1536,13,2,7)$ was constructed by Fon-Der-Flaass 
in \cite{FDF:PerfCol},
in terms of equitable partitions,
as a special case of a general construction, 
and OA$(768,12,2,6)$ is obtained
from  OA$(1536,13,2,7)$ by shortening.
The orthogonal-array (correlation-immune) properties of equitable partitions were 
known, see e.g. \cite{FDF:CorrImmBound},
but not mentioned in \cite{FDF:PerfCol},
and the construction in that paper was unnoticed 
by specialists in orthogonal arrays for some time, see, e.g., 
Table~5 in \cite{BMS:2017:few},
where these parameters are still marked as unsolved.

Utilizing local properties of equitable partitions, 
we use exhaustive computer search
to find all OA$(1536,13,2,7)$ up to equivalence 
and establish that there is only one
equivalence class of such orthogonal arrays.
A similar approach was already used in \cite{KroVor:CorIm} to characterize
the orthogonal arrays OA$(1024, 12, 2, 7)$.
However, the direct generalization 
of the approach of \cite{KroVor:CorIm} 
did not work for OA$(1536,13,2,7)$,
and in this paper we modify it,
considering the value of the first coordinate separately from the others.
All calculations took about one core-year of computing on a 2GHz processor.
% :
% the instances of the exact-cover problem to be solved at each step were too large to
% solve them by modern software.
% To make them smaller, we adopt the approach
% of \cite{KroVor:CorIm} in Section~\ref{s:class} 
% and divide the vertices of some weight
% into two groups, 
% depending on the value of the first coordinate.
% As a result we divide the classification into more steps and reduce
% the size of the exact-cover problem at each step.
% This modified approach had a success, 
% and we got the classification in only one core-year 
% of computing time. 
By shortening the unique OA$(1536,13,2,7)$, we find all inequivalent OA$(768,12,2,6)$.

The small parameters of binary 
orthogonal arrays attending the Friedman bound $N\ge 2^n(1-n/2(t+1))$
\cite[Theorem~2.1]{Friedman:92}
and satisfying the Fon-Der-Flaass--Khalyavin bound
$t\le 2n/3-1$ \cite{FDF:CorrImmBound,Khalyavin:2010.en}
are shown in Table~\ref{t:1}.
% (we omit the parameters that attends the  Friedman bound, but do not fit the Fon-Def-Flaass--Khalyavin bound
% $t\le 2n/3-1$ for $N < 2^n$ \cite{FDF:CorrImmBound,Khalyavin:2010.en}, e.g., OA$(640,11,2,7)$).
\begin{table}[ht]
$$
\begin{tabular}{lll}
OA & quotient matrix & {number of equivalence classes} \\ \hline
 OA$(2,3,2,1)$ & $[[0,3],[1,2]]$ & $1$  \\
 OA$(16,6,2,3)$ & $[[0,6],[2,4]]$ & $1$   \\
 OA$(16,7,2,3)$ & $[[0,7],[1,6]]$ & $1$ \quad (the Hamming $(7,16,3)$ code) \\
 OA$(128,9,2,5)$ & $[[0,9],[3,6]]$ & $2$ \quad (see \cite{Kirienko2002}) \\
 OA$(1024,12,2,7)$ & $[[0,12],[4,8]]$ & $16$ \quad (see \cite{KroVor:CorIm}) \\
 OA$(1536,13,2,7)$ & $[[0,13],[3,10]]$ & $1$ \quad (Theorem~\ref{th:13}) \\
 OA$(2048,14,2,7)$ & $[[0,14],[2,12]]$ & $>14960$ \quad (Corollary~\ref{c:14}) \\ 
 %punctured $1$-perfect codes and other \\
 OA$(2048,15,2,7)$ & $[[0,15],[1,14]]$ & 5983 \quad ($1$-perfect $(15,2^{11},3)$ codes \cite{OstPot:15}) \\ 
  OA$(8192,15,2,9)$ & $[[0,15],[5,10]]$ & ?  \\ \hline 
\end{tabular}
$$
\caption{A list of small parameters OA$(N,n,2,t)$ of binary orthogonal arrays satisfying
$N= 2^n(1-n/2(t+1))$, $t \le 2n/3-1$, and the corresponding equitable partitions}
\label{t:1}
\end{table}
The first unclassified case is OA$(2048,14,2,7)$.
It can be shown that puncturing (projecting in one coordinate)
 any $1$-perfect binary code of length $15$ gives an orthogonal array 
with these parameters.
Such codes, of parameters $(15,2^{11},3)$, were classified by \"Osterg{\aa}rd and Pottonen 
\cite{OstPot:15}, and all the punctured codes can be derived from that classification.
In Section~\ref{s:14}, we show that this is not enough for the complete classification
of OA$(2048,14,2,7)$: there are such orthogonal arrays that are not punctured $1$-perfect codes.

The parameters OA$(768,12,2,6)$ attend the general theoretical bound
$N\ge 2^n - 2^{n-2}(n+1)/\lceil (t+1)/2 \rceil $ \cite{BGS:96} for 
OA$(N,n,2,t)$. As a theoretical contribution in addition to the computational results, 
in this paper 
we prove that the orthogonal arrays attending this bound 
are in one-to-one correspondence 
with the equitable $3$-partitions with a special quotient matrix.
Thus, one more family is added to the collection of classes of optimal objects
(e.g., perfect and nearly perfect codes, 
some other classes of codes
\cite{Kro:2m-3,Kro:2m-4},
orthogonal arrays \cite{Pot:2012:color}, 
correlation-immune functions \cite{FDF:CorrImmBound})
whose parameters guarantee that they can be described in terms of equitable partitions.

The paper is organized as follows.
In the next section, we define the basic concepts,
mainly related with orthogonal arrays and equitable partitions,
and mention some basic theoretical facts.
In Section~\ref{s:oa-eq}, we consider 
known and new (Theorem~\ref{th:3part})
general theoretical results connecting
equitable partitions and orthogonal arrays.
In Section~\ref{s:class}, we describe the classification approach.
The results of the classification of the orthogonal arrays OA$(1536,13,2,7)$ and OA$(768,12,2,6)$ can be found in Section~\ref{s:res}.
In Section~\ref{s:constr}, we describe the unique  OA$(1536,13,2,7)$ 
in two ways, by the Fon-Der-Flaass construction and by the Fourier transform.
Section~\ref{s:14} is devoted to the orthogonal arrays  OA$(2048,14,2,7)$
and arrays with similar parameters.
In the concluding section, 
we highlight some open research problems.

%===============================================
%===============================================
%===============================================
\section{Definitions}\label{s:def}

\begin{definition}[graphs and related concepts]
A (simple) \emph{graph} is a pair $(V,E)$ of a set $V$, 
whose elements are called \emph{vertices},
and a set $E$ of $2$-subsets of $V$, 
called \emph{edges}. 
Two vertices in the same edge are called \emph{neighbor}, or \emph{adjacent},
to each other. The number of neighbors of a vertex is referred to as its \emph{degree}.
A graph whose vertices have the same degree is called \emph{regular}.
An \emph{isomorphism} between two graphs is a bijection between their vertices that induces
a bijection between the edges.
Two graphs are \emph{isomorphic} if there is an isomorphism between them.
An \emph{automorphism} of a graph is an isomorphism to itself.
A set of vertices of a graph is called \emph{independent} if it does not 
include any edge.
\end{definition}

\begin{definition}[Hamming graphs and related concepts]
The \emph{Hamming graph} $H(n,q)$
is a graph whose vertex set is the set $\{0, 1,\ldots,q-1\}^n$ 
of the words of length $n$ over the alphabet 
$\{0, 1,\ldots,q-1\}$.
Two vertices are adjacent if and only if they differ in exactly one coordinate position,
which is referred to as the \emph{direction} of the corresponding edge.
The \emph{Hamming distance} $d(\bar x,\bar y)$ between vertices $\bar x$ and $\bar y$ 
is the number of coordinates in which they differ.
The \emph{weight} $\wt(\bar x)$ of a word $\bar x$ 
is the number of nonzero elements in it.
In this paper, we focus on the binary Hamming graph $H(n,2)$,
also known as the \emph{$n$-cube} $\Qq{n}$. 
The vertices of $Q_n$ are also considered as vectors over the $2$-element field
GF$(2)$, with the coordinate-wise addition and multiplication by a constant.

For two words $\bar u$ and $\bar v$, we denote by $\bar u|\bar v$ their concatenation.
The all-zero word and the all-one word are denoted by $\bar 0$ and $\bar 1$, respectively
(the length is usually clear from the context).
\end{definition}

\begin{definition}[orthogonal arrays and related concepts]
An \emph{orthogonal array} OA$(N,n,q,t)$ is a multiset $C$
of vertices of $H(n,q)$ of cardinality $N$ such that every 
subgraph isomorphic to $H(n-t,q)$ contains exactly
$N/q^{t}$ elements of $C$.
An {orthogonal array} is \emph{simple} if it is a usual set; that is,
if it does not contain elements of multiplicity more than $1$.
% A (Boolean) function $\{0,1\}^n\to \{0,1\}$ is called \emph{$t$-th order correlation immune} 
% if the set of vectors with function value $1$ is an OA$(N,n,2,t)$ 
% (so, the $t$-th order correlation immune functions are in one-to-one correspondence 
% with the simple binary (i.e., $q=2$) orthogonal arrays of strength $t$).
Two orthogonal arrays are \emph{equivalent} if some automorphism of $H(n,q)$
induces a bijection between their elements.
The \emph{automorphism group} $\mathrm{Aut}(C)$ of an orthogonal array $C$
(as well as any other set $C$ of vertices of $H(n,q)$)
consists of all automorphisms of $H(n,q)$ that stabilize $C$ set-wise.
The \emph{orbit} of a vertex $\bar v$ under the action of  $\mathrm{Aut}(C)$ is the 
set of images $A(\bar v)$ over all $A$ from $\mathrm{Aut}(C)$.
The \emph{kernel} of $C\subseteq \{0,1\}^n$ is the set $\{\bar k\in \{0,1\}^n \,:\, \bar k+C=C\}$ 
of all its periods.
For any set $C$ of vertices of $H(n,q)$, by $\overline C$ we denote its complement
$\{0, 1,\ldots,q-1\}^n \backslash C$. Obviously, the complement
of a simple  OA$(N,n,q,t)$ is a simple OA$(q^n-N,n,q,t)$.
We say that an orthogonal array $C'$ is obtained from an  orthogonal array $C$
by \emph{$a$-shortening}, or simply by \emph{shortening}, 
in the $i$-th position (by default, in the last position)
if $C'$ is obtained from $C$ by choosing all words with a symbol $a\in\{0,\ldots,q-1\}$ 
in the $i$-th position and removing this symbol in this position from all chosen words.
\end{definition}

\begin{remark}
In the current paper, motivated by the coding-theory 
technique used in this research,
we consider orthogonal arrays as sets of vertices
of the Hamming graph (for generality, they are defined as multisets, 
but all considered arrays are simple). 
The definition above is equivalent to the traditional 
definition mentioned in the introduction, 
but uses a different language.
According to our definition, the elements of the array are words, corresponding
to the rows (\emph{runs}) in the traditional definition, 
and the positions in the words, or the coordinates, numbered from $1$ to $n$,
correspond to the columns (\emph{factors}) in the traditional definition.
 In the classical literature on orthogonal arrays,
%  they are often represented as $N\times n$ or $n\times N$ arrays,
%  the elements of the multiset corresponding 
%  to the rows or the columns of 
%  the array and being referred to as the \emph{runs}.
 the parameters $N$, $n$, $q$, $t$, $\lambda=N/q^{t}$
 are known as respectively
 the number of (experimental) runs, 
 the number of factors,
 the number of levels, 
 the strength, 
 and the index of the orthogonal array.
\end{remark}

\begin{definition}[equitable partitions]
Let $G=(V,E)$ be a graph. 
A partition $(C_0,\ldots, C_{k-1})$
of the set $V$ is an \emph{equitable partition} 
(in some literature, 
\emph{regular partition}, 
\emph{perfect coloring},
or  \emph{partition design}),
or \emph{equitable $k$-partition},
with the 
\emph{quotient matrix}
$S=(s_{ij})$ if for all $i$ and $j$ from $\{0, \ldots, k-1\}$ 
every vertex of $C_i$ has exactly $s_{ij}$ neighbors in $C_j$.
Below, for convenience, a $2\times 2$ quotient matrix will be
represented by its row list: 
$ 
\left(\begin{array}{cc}a&b\\c&d\end{array}\right)=[[a,b],[c,d]].
$
\end{definition}

%==============================================
%==============================================
%==============================================
\section{Connections between orthogonal arrays and equitable partitions}\label{s:oa-eq}

The following folklore fact
establishes the orthogonal-array properties of an equitable partition.
For equitable $2$-partitions of $\Qq{n}$ (the case we focus on), 
it can be found, e.g., in~\cite{FDF:CorrImmBound}.
\begin{pro} \label{p:equit-oa}
 Each cell $C$ of an equitable partition of $H(n,q)$ with quotient matrix $S$
is a simple
 OA$(|C|,n,q,t)$, 
 where $t=\frac{n(q-1)-\theta}{q} - 1$
 and $\theta$ is the second largest eigenvalue of $S$.
 In particular, $t=\frac{b+c}{q} - 1$ if $S=[[a,b],[c,d]]$.
\end{pro}
\begin{proof}[(a sketch)]
 It follows from the general theory of equitable partitions \cite{Gogsil:93:equitable} 
 that the characteristic 
 $\{0,1\}$-function of every cell of the equitable partition can be represented as the sum of eigenfunctions of the graph
 with eigenvalues that coincide with eigenvalues of the quotient matrix $S$.
 The eigenspaces of the Hamming graph have very convenient bases from so-called \emph{characters} 
 (see Section~\ref{s:Fr} for the definition in the binary case).
 It is straightforward to check that for every character corresponding to a non-largest eigenvalue
 of $S$, the sum of values over the vertices of a subgraph isomorphic to $H(n-t,q)$ is $0$.
 For the largest eigenvalue, an eigenfunction is a constant function. 
 This means that $C$ 
 has a constant number of vertices 
 in every subgraph isomorphic to $H(n-t,q)$; 
 i.e., it is an OA$(|C|,n,q,t)$.
\end{proof}

In some cases, 
the parameters of an orthogonal array guarantee that it is a cell
of an equitable partition.
One of the known bounds on the parameters of orthogonal arrays,
proved by Friedman \cite[Theorem~2.1]{Friedman:92} for the binary
case $q=2$ and by Bierbrauer \cite{Bierbrauer:95}
for an arbitrary $q$, says that the size $N$ of
an OA$(N,n,q,t)$ satisfies the inequality
\begin{equation}\label{eq:bier}
 N\ge q^n\left(1-\frac{(q-1)n}{q(t+1)}\right).
\end{equation}
As follows from the proof, see \cite[p.\,181, line 4]{Bierbrauer:95},
the inequality is strict for non-simple arrays (with repeated elements).
Moreover, an orthogonal array that attains this bound
is an independent set and
forms an equitable $2$-partition,
in the pair with its complement.

\begin{pro}[\cite{Potapov:2010} ($q=2$), 
\cite{Pot:2012:color}]
\label{p:oa-equit}
If \eqref{eq:bier} holds with equality for some OA$(N,n,q,t)$ $C$,
then $(C,\overline C)$ is an equitable partition with quotient matrix
$$
\left(\begin{array}{cc}
0 & (q{-}1)n \\ q(t{+}1){-}(q{-}1)n   & 2(q{-}1)n{-}q(t{+}1)
      \end{array}\right),
\ \mbox{in particular, }
\left(\begin{array}{cc}
0 & n \\ 2(t{+}1){-}n   & 2n{-}2(t{+}1)
      \end{array}\right)
\mbox{ if $q=2$}.
$$
\end{pro}

So, by Propositions~\ref{p:equit-oa} and~\ref{p:oa-equit}, 
there is a bijection between the orthogonal arrays attaining the Bierbrauer--Friedman bound
and the equitable $2$-partitions of the Hamming graph with 
the first coefficient of the quotient
matrix being $0$.

Next, we consider a bound for binary orthogonal arrays of even strength,
which follows straightforwardly from the Friedman bound and the following fact
about lengthening a binary array of even strength
(this fact can be considered as a dual analog of the possibility of extending
a binary $(n,M,2e+1)$ code to an $(n+1,M,2e+2)$ code, well known in the theory 
of error-correcting codes, see e.g. \cite[1.9(I)]{MWS}).

\begin{pro}[{\cite[Proposition~2.3]{SeiZem:OA:66}}]\label{p:short}
If $t$ is even, then every orthogonal array 
OA$(N,n,2,t)$ can be obtained from some 
OA$(2N,n+1,2,t+1)$ by shortening.
Specifically, if $C$ is an OA$(N,n,2,t)$,
then $C|0 \cup C'|1$, where $C'=C+\bar 1$, is an OA$(2N,n+1,2,t+1)$.
\end{pro}

As was noted by V.Levenshtein 
(cited in \cite{BGS:96} as a private communication), 
Proposition~\ref{p:short}
with \eqref{eq:bier} imply the inequality
\begin{equation}\label{eq:lev}
 N\ge 2^n\left(1-\frac{n+1}{2(t+2)}\right)
\end{equation}
for the parameters of a binary orthogonal array OA$(N,n,2,t)$ of even strength $t$.
We can note that every orthogonal array attending  bound \eqref{eq:lev}
is a cell of an equitable $3$-partition.

\begin{theorem}\label{th:3part}
 Assume that $C$ is an orthogonal array OA$(N,n,2,t)$ of even strength $t$ meeting 
 \eqref{eq:lev} with equality. 
 If $C'=C+\bar 1 = \{ \bar c + \bar 1 \,:\, \bar c \in C \}$ and 
 $C''=\{0,1\}^n\backslash (C \cup C') $,
 then
 $(C, C', C'')$ is an equitable partition 
 with quotient matrix 
 \begin{equation}\label{eq:qm3}
 \left(
 \begin{array}{ccc}
  0 & 2t{-}n{+}2 & 2n{-}2t{-}2 \\
  2t{-}n{+}2 & 0 & 2n{-}2t{-}2 \\
  2t{-}n{+}3 & 2t{-}n{+}3 & 3n{-}4t{-}6 \\
 \end{array}
 \right)
 =
 \left(
 \begin{array}{ccc}
  0 & a & n{-}a \\
  a & 0 & n{-}a \\
  a{+}1 & a{+}1 & n{-}2a{-}2 \\
 \end{array}
 \right)
 ,
 \quad \mbox{where $a=2t-n+2$.}
 \end{equation}
\end{theorem}

\begin{proof}
Denote $C'=C+\bar 1$ and  $B=C|0 \cup C'|1$.
 By Proposition~\ref{p:short}, $B$
 is an orthogonal array 
 OA$(2N,n+1,2,t+1)$.
 By Proposition~\ref{p:oa-equit},
 $(B,\overline B)$ is an equitable partition with quotient matrix 
 $$
\left(\begin{array}{cc}
0 & n+1 \\ 2t{-}n{+}3   & 2n{-}2t{-}2
      \end{array}\right)
      =
      \left(\begin{array}{cc}
0 & n+1 \\ a{+}1   & n{-}a
      \end{array}\right),
      \quad
      \mbox{where $a=2t-n+2$.}
$$
 (note that $n$ and $t$ in Proposition~\ref{p:oa-equit} 
 correspond respectively to $n+1$ and $t+1$ in our case).
 Denoting $C''=\{0,1\}^n \backslash (C \cup C')$,
 we observe 
 that $(C,C',C'')$ is a partition of $\{0,1\}^n$ 
 (indeed, $C$ and $C'$ are disjoint because $B$ is an independent set).
 It remains to check that it is an equitable partition.
 
 Consider a vertex $\bar v$ from $C$. 
 As $B$ is an independent set, $C$ is an independent set too,
 and $\bar v$ has $0$ neighbors in $C$.
 Moreover, $\bar v|0$ from $B$ has no neighbors in $B$,
 and $\bar v|1$ from $\overline B$ has $a+1$ neighbors in $B$;
 one of them is  $\bar v|0$ and the other $a$ are in $C'|1$.
 Hence, $\bar v$ has $a$ neighbors in $C'$.
 The other neighbors of $\bar v$ are in $C''$, 
 and the first row of the quotient matrix  \eqref{eq:qm3} is confirmed.
 The second row is similar.
 For the third row, consider
  a vertex $\bar u$ from $C''$. Both $\bar u|0$ and $\bar u|1$ are in $\overline B$.
  Each of them has $a+1$ neighbors in $B$, but those neighbors of $\bar u|0$ are
  in $C|0$, while those neighbors of $\bar u|1$ are
  in $C'|1$. So, $\bar u$ has exactly $a+1$ neighbors in $C$ and exactly $a+1$ neighbors in $C'$;
  the third row of the quotient matrix is confirmed.
%   It remains to show that $C'=C+\overline 1$. But this is equivalent to 
%   $B=B+\overline 1$, which is a well-known antipodal property of the 
%   equitable $2$-partitions of the hypercube with non-symmetric quotient matrix
%   (a direct corollary of distance-invariant properties of equitable partitions, 
%   see e.g. \cite{Kro:struct}).
\end{proof}

\begin{remark}
The equitable partitions with quotient matrices \eqref{eq:qm3} are connected (in a one-to-one manner)
 with a special class of
 completely regular codes. A \emph{completely regular code} of covering radius $\rho$ 
 is the first cell of an equitable $(\rho+1)$-partition with a tridiagonal quotient matrix.
 We start from an equitable partition $(C,C',C'')$ with quotient matrix \eqref{eq:qm3}.
Divide each of $C$, $C'$ into two subsets, 
 respectively
 $C_{\mathrm{even}}$ and  $C_{\mathrm{odd}}$,
 $C'_{\mathrm{even}}$ and  $C'_{\mathrm{odd}}$,
 according to the parity of the weight of vertices.
 It is straightforward to check that
 $(C_{\mathrm{even}} \cup C'_{\mathrm{odd}},C'',C'_{\mathrm{even}} \cup C_{\mathrm{odd}})$
 is an equitable partition with quotient matrix 
 $$
 \left(
 \begin{array}{ccc}
  a& n{-}a & 0  \\
  a{+}1 & n{-}2a{-}2 & a{+}1 \\
  0& n{-}a & a  \\
 \end{array}
 \right).
 $$
 So, $C_{\mathrm{even}} \cup C'_{\mathrm{odd}}$ (as well as $C'_{\mathrm{even}} \cup C_{\mathrm{odd}}$) is a completely regular code.
\end{remark}

The parameters of orthogonal arrays OA$(1536,13,2,7)$ lie on 
the Bierbrauer--Friedman bound \eqref{eq:bier}.
It is straightforward to see that the corresponding quotient matrix  
is \linebreak[4]
$[[0,b],[c,d]]=[[0,13],[3,10]]$ (indeed,
$0+b=c+d=13$ and $c:(b+c)=1536:2^{13}$);
equitable partitions with this quotient matrix are known to exist
\cite[Proposition~2]{FDF:PerfCol}.
One of the main results of the current research is establishing that 
the OA$(1536,13,2,7)$
constructed in \cite{FDF:PerfCol} in terms of perfect colorings (equitable partitions)
is unique up to equivalence.
The related parameters OA$(768,12,2,6)$ attend bound \eqref{eq:lev}, 
and correspond (Theorem~\ref{th:3part}) to the quotient matrix $[[0,2,10],[2,0,10],[3,3,6]]$;
the characterization of such orthogonal arrays is derived from the uniqueness
of OA$(1536,13,2,7)$.

In the end of this section, for completeness, 
we mention another interesting bound that relates 
orthogonal arrays with equitable partitions.
Fon-Der-Flaass proved \cite{FDF:CorrImmBound}
that any simple OA$(N,n,2,t)$ such that $N \not\in\{0,2^{n-1}, 2^n\}$
% (or, equivalently, a non-constant non-balanced Boolean function of
% correlation-immunity order $t$)
satisfies $t\le \frac{2n}{3}-1$, and in the case of equality, the OA is a cell
of an equitable $2$-partition of $\Qq{n}$.
For example, simple orthogonal arrays of parameters OA$(1792,12,2,7)$ correspond to equitable partitions with quotient matrix
$[[3,9],[7,5]]$, constructed in \cite{FDF:12cube.en}.
Later, Khalyavin \cite{Khalyavin:2010.en} proved that this
bound  also holds if $0 < N < 2^{n-1}$ and we do not require the 
orthogonal array to be simple. 
However,
in contrast to the case of the Bierbrauer--Friedman bound,
the orthogonal arrays on the
 Fon-Der-Flaass--Khalyavin bound 
are not necessarily simple
(e.g., there exists a non-simple OA$(24,6,2,3)$ \cite{Taran:private2018}),
and hence not connected with equitable partitions in general.

%====================================================
%====================================================
%====================================================
\section[Classification of OA(1536,13,2,7)]{Classification of OA$(1536,13,2,7)$}\label{s:class}

For classification by exhaustive search, 
we use an approach based 
on the local properties of the equitable partitions. 
Say that the pair of disjoint sets $P_+$,
$P_-$ 
of vertices of $\Qq{13}$ is an \emph{$(r_0,r_1)$-local partition}
if 
\begin{itemize}

 \item[(I)] 
 $P_+ \cup P_-$ 
are the all words starting with $0$
and having weight at most $r_0$
or starting with $1$
and having weight at most $r_1$;

\item[(II)] 
$P_+$ contains the all-zero word $\bar 0$;

\item[(III)] 
$P_+$ is an independent set;
\item[(IV)] 
the neighborhood of every vertex 
$\bar v=(v_1,\ldots,v_{13})$
of weight less than $r_{v_1}$ satisfies the local condition 
from the definition 
of an equitable partition 
with quotient matrix 
$[[0,13],[3,10]]$ 
(that is, if $\bar v\in P_+$ then 
the whole neighborhood of $\bar v$
is included in $P_-$;
if $\bar v\in P_-$ then
the neighborhood has exactly $3$
elements in $P_+$ and $10$ in $P_-$).
\end{itemize}
Two $(r_0,r_1)$-local partitions
$(P_+,P_-)$ and $(P'_+,P'_-)$
are  \emph{equivalent}
if there is a permutation
of coordinates that fixes the first
coordinate and sends
$P_+$ to $P'_+$.

We classify all inequivalent 
$(r_0,r_1)$-local partitions subsequently
for $(r_0,r_1)$ equal
$(2,2)$, $(2,3)$, $(2,4)$, $(3,4)$, $(4,4)$, $(4,5)$, $(5,5)$, $(13,13)$, where 
$(13,13)$ corresponds to the complete equitable partitions.
In an obvious way, every equitable partition $(C,\overline C)$ such that $\bar 0\in C$ includes
a $(5,5)$-local partition $(P_+^{(5,5)},P_-^{(5,5)})$, 
$P_+^{(5,5)} \subset C$ and $P_-^{(5,5)} \subset \overline C$, 
every $(5,5)$-local partition $(P_+^{(5,5)},P_-^{(5,5)})$ includes
a $(4,5)$-local partition $(P_+^{(4,5)},P_-^{(4,5)})$, 
$P_+^{(4,5)} \subseteq P_+^{(5,5)}$ and $P_-^{(4,5)} \subseteq P_-^{(5,5)}$, 
and so on.
So, the strategy is to reconstruct, 
in all possible ways, 
a $(r_0,r_1)$-local partition from
each of the inequivalent $(r_0-1,r_1)$-local 
or $(r_0,r_1-1)$-local partitions,
and then to choose and keep only inequivalent solutions,
one representative for each equivalence class found.
Our classification is divided into the following steps.

\begin{itemize}

 \item[1.] (Section \ref{s:22}.) 
 Manually characterizing the $(2,2)$-local partitions, up to equivalence.
 
 \item[2.] (Section \ref{s:22-55}.) 
 Characterizing, up to equivalence,
the $(2,3)$-local partitions 
 based on the known representatives of
  $(2,2)$-local partitions 
  and
  using the exact-covering software \cite{KasPot08}.  
  Similarly,
  from $(2,3)$ to $(2,4)$,
  from $(2,4)$ to $(3,4)$,
  from $(3,4)$ to $(4,4)$,
  from $(4,4)$ to $(4,5)$,
  from $(4,5)$ to $(5,5)$.
  Equivalence is recognized using the graph-isomorphism software \cite{nauty2014}. The results (see Table~\ref{tab:1}) are validated by
  double-counting using the orbit-stabilizer theorem. 
  % The calculation took about one core-year on a 2GHz computer.
  
  \item[3.] (Section \ref{s:13}.) Reconstructing an equitable partition from a $(5,5)$-local partition.
  It follows from the definition of orthogonal arrays that 
  a complete equitable partition
  can be reconstructed in a unique way.
\end{itemize}

\begin{table}[ht]
\caption{The number of equivalence classes of $(r_0,r_1)$-local partitions classified by the type of 
the included $(2,2)$-local partition}
$$
\begin{array}{l|c|c|c|c|c|c}
\mbox{type} & (r_0,r_1)=(2,3) & (2,4) & (3,4) & (4,4) & (4,5) & (5,5) \\ \hline
4{+}3{+}3{+}3 & 266 & 33077 & 912 & 0 &  &    \\
3{+}4{+}3{+}3 & 475 & 97550 & 187335 & 0 &  &    \\
7{+}3{+}3 & 2315 & 861699 & 97841 & 0 &  &    \\
3{+}7{+}3 & 2540 & 839273 & 1198056 & 0 &  &    \\
6{+}4{+}3 & 3492 & 1362844 & 37234 & 0 &  &    \\
4{+}6{+}3 & 4134 & 748748 & 3724 & 0 &  &    \\
3{+}6{+}4 & 2404 & 861732 & 452111 & 0 &  &    \\
5{+}5{+}3 & 2611 & 1194122 & 69325 & 10 & 20 & 20   \\
3{+}5{+}5 & 1156 & 444846 & 330614 & 12 & 12 & 12   \\
10{+}3 & 25784 & 11598959 & 699031 & 14 & 20 & 20   \\
3{+}10 & 10579 & 4336586 & 3656845 & 19 & 15 & 12   \\
5{+}4{+}4 & 1397 & 565938 & 7864 & 0 &  &    \\
4{+}5{+}4 & 3785 & 701873 & 1192 & 0 &  &    \\
9{+}4 & 19809 & 9262166 & 186257 & 0 &  &    \\
4{+}9 & 15802 & 3240956 & 9203 & 0 &  &    \\
8{+}5 & 15149 & 7843990 & 229791 & 0 &  &    \\
5{+}8 & 9518 & 5006596 & 147247 & 0 &  &    \\
7{+}6 & 12777 & 6436913 & 185167 & 0 &  &    \\
6{+}7 & 10901 & 5446544 & 124577 & 0 &  &    \\
13 & 150346 & 77748861 & 2425510 & 0 &  &    \\ \hline
 \mbox{any} &  295240 & 138633273 & 10049836 & 55 & 67 & 64
\end{array}
$$
\label{tab:1}
\end{table}

%====================================================
\subsection[The (2,2)-local partitions]{The $(2,2)$-local partitions}\label{s:22}

A starting point of our classification 
is the $(2,2)$-local partitions, 
which can be classified manually.

\begin{lemma}
 There are exactly $20$ equivalence classes of $(2,2)$-local partitions.
\end{lemma}
\begin{proof}
Assume that $(P_+,P_-)$ is a $(2,2)$-local partition. By the definition,
$\bar 0 \in P_+$. Moreover, all $13$ weight-$1$ words belong to $P_-$.
Each of them has $3$ neighbors in $P_+$, by the definition of an equitable partition.
One of these $3$ neighbors is $\bar 0$,
while the other two have weight $2$. 
On the $13$ weight-$1$ words, we construct a graph $\Gamma_{13}$, 
two vertices being connected if and only if they have a common weight-$2$ neighbor in $P_+$.
We see that this graph is regular of degree $2$ (i.e., a $2$-factor, consisting of disjoint cycles), 
and it completely determines $P_+$ and hence $P_-$.
There are $10$ such graphs, up to isomorphism, with cycle structures
$4{+}3{+}3{+}3$, $7{+}3{+}3$, $6{+}4{+}3$, 
$5{+}5{+}3$, $10{+}3$, $5{+}4{+}4$, 
$9{+}4$, $8{+}5$, $7{+}6$, and $13$.
However, two isomorphic graphs correspond to inequivalent $(2,2)$-local partitions
if and only if the weight-$1$ word with $1$ in the first coordinate belongs to cycles of different length in these two graphs.
So, inequivalent $(2,2)$-local partitions correspond to non-isomorphic pairs
(a $2$-factor on $13$ vertices, a chosen vertex). 
There are exactly $20$ such  non-isomorphic pairs, with the cycle structures
$\dot3{+}4{+}3{+}3$, 
$\dot4{+}3{+}3{+}3$, 
$\dot3{+}7{+}3$, 
$\dot7{+}3{+}3$, 
$\dot3{+}6{+}4$, 
$\dot4{+}6{+}3$, 
$\dot6{+}4{+}3$,
$\dot3{+}5{+}3$, 
$\dot5{+}5{+}3$, 
$\dot3{+}10$, 
$\dot{10}{+}3$, 
$\dot4{+}5{+}4$,
$\dot5{+}4{+}4$,
$\dot4{+}9$, 
$\dot9{+}4$, 
$\dot5{+}8$, 
$\dot8{+}5$, 
$\dot6{+}7$, 
$\dot7{+}6$, 
and $\dot{13}$,
where the first (dotted) summand 
corresponds to the length of the cycle
that contains the chosen vertex.
\end{proof}

\subsection[From (2,2) to (2,3), (2,4), \ldots, (5,5)]{From $(2,2)$ to $(2,3)$, $(2,4)$, \ldots, $(5,5)$}\label{s:22-55}
We describe these steps by the example of the case $(2,3) \to (2,4)$,
as the other cases are completely similar and solved with the same
\texttt{c++} program with different parameters.

%====================================================
\subsubsection[Completing to (2,4)-local partition]{Completing to $(2,4)$-local partition}\label{s:23-24}

Denote by $W_{i}^{j}$ the set of words of weight $j$ that start with $i$.
As the result of the previous step,
we keep representatives of all the equivalence
classes of $(2,3)$-local partitions.
For each representative $(P_+,P_-)$, 
we need to find a subset $R$ of $W_{1}^{4}$ such that
$(P_+ \cup R,P_-\cup (W_{1}^{4} \backslash R))$ 
is a  $(2,4)$-local partition, i.e., satisfies (I)--(IV).

Conditions (I) and (II) are satisfied automatically. 
To satisfy condition (III), 
we remove from $W_{1}^{4}$ all the words that have a neighbor
from $P_+$. 
The set obtained, call it $U$, 
is the set of candidates 
for the role of elements of $R$. 
It remains to satisfy condition (IV) for all 
the vertices from $W_{1}^{3}\cap P_-$ 
(for the vertices from $P_+$, 
it is satisfied by (III);
for the vertices from $W_0^0$, 
$W_0^1$, $W_1^1$, and $W_1^2$, 
it is satisfied because of the $(2,3)$-local property).
For each vertex $\bar u$ from $W_{1}^{3}\cap P_-$, 
denote $\alpha(\bar u)=3-\beta(\bar u)$,
where $\beta(\bar u)$ is the number of neighbors of $\bar u$ in $P_+$.
By the definition of a $(2,4)$-local partition, 
$\bar u$ must have exactly 
$\alpha(\bar u)$ neighbors in $R$.
So, to meet (IV), we have to find a collection $R$
of elements from $U$ such that every element
$\bar u$ from $W_{1}^{3}\cap P_-$
belongs to exactly $\alpha(\bar u)$
neighbors of elements of $R$.
This is an instance of the problem known as \emph{exact covering}.
A convenient package to solve this problem
(with different multiplicities $\alpha(\bar u)$,
which is important in our case)
in \texttt{C} and \texttt{C++} programs is \texttt{libexact} \cite{KasPot08}.
After finding all the solutions $R$, 
we have all the $(2,4)$-local partitions
that include the given $(2,3)$-local partition $(P_+,P_-)$.

%====================================
\subsubsection{Isomorph rejection}\label{s:iso}
As we need to keep only inequivalent $(2,4)$-local partitions,
it is important to compare such partitions for equivalence.
It is done with the help of the well-known graph-isomorphism software \cite{nauty2014}.
The standard technique, described in \cite{KO:alg}, 
consists of constructing 
for each object
(in our case, a $(2,4)$-local partition)
a \emph{graph} such that two objects are equivalent 
if and only if the corresponding graphs are isomorphic.
Using the \texttt{nauty\&traces} package  \cite{nauty2014},
from each graph we can construct the \emph{canonical-labeling graph}
such that two graphs are isomorphic if and only if
the corresponding canonical-labeling graphs are equal.
Each time we find a new $(2,4)$-local partition,
we construct the canonical-labeling graph $G$ and check 
whether it is contained in our collection (of inequivalent $(2,4)$-local partitions
and the corresponding canonical-labeling graphs).
If not, we update the collection with the new representative
and the corresponding canonical-labeling graph $G$,
and set the value of a special variable $N(G)$,
the number of occurrences, 
equal to $1$
(the final value of $N(G)$
is utilized in the validation step,
see the next subsection).
If the graph $G$ is already in the collection, we only increase $N(G)$ by $1$.
When the search is finished, our collection contains representatives
of all the equivalence classes of $(2,4)$-local partitions.

%====================================
\subsubsection{Validation}\label{s:valid}

We can validate the results of the calculation by double-counting 
the size of each equivalence class found. Let $(P'_+,P'_-)$
be a $(2,4)$-local partition, and let it include a
$(2,3)$-local partition $(P_+,P_-)$.
On one hand, there are exactly
\begin{equation}
 \label{eq:N1}
 \frac{12!}{|\mathrm{Aut}(P'_+,P'_-)|} 
\end{equation}
$(2,4)$-local partitions equivalent to $(P'_+,P'_-)$,
where $\mathrm{Aut}(P'_+,P'_-)$ is the set of permutations of the last $12$ 
coordinates that stabilize $P'_+$ and $P'_-$ set-wise. 
On the other hand, this number equals
\begin{equation}
 \label{eq:N2}
 N(P'_+,P'_-)\cdot\frac{12!}{|\mathrm{Aut}(P_+,P_-)|},
\end{equation}
where $N(P'_+,P'_-)$ is the number of $(2,4)$-local partitions 
that are equivalent to $(P'_+,P'_-)$ and include the $(2,3)$-local partition $(P_+,P_-)$.
As our algorithm finds all $(2,4)$-local partitions 
that include a given $(2,3)$-local partition,
the number $N(P'_+,P'_-)$ for each found equivalence class
is computed during the isomorph rejection step and equals 
the final value of $N(G)$ for the corresponding graph $G$,
see the previous subsection.
If the number $N(P'_+,P'_-)$ is counted correctly, 
then we know that we did not miss any representative
of the equivalence class during the experiment.
So, calculating the values \eqref{eq:N1} and \eqref{eq:N2}
and comparing them for equality
prevents many kinds of random and systematical errors.
This strategy represents a special case of 
the general double-counting validation technique 
described in~\cite[10.2]{KasPot08}.
Note that $|\mathrm{Aut}(P'_+,P'_-)|$
coincides with the order of the automorphism group of the corresponding 
characteristic graph; it is computed by  \texttt{nauty\&traces}
as a part of finding the canonical-labeling graph.

%====================================================
\subsection{Completing to an equitable partition}\label{s:13}

Completing a $(5,5)$-local partition
$(P_+,P_-)$ to an equitable partition 
$(C,\overline C)$
of $\Qq{13}$ is the easiest step,
and the result is always unique
(however, 
the fact that it always exists is still only empiric).
We know that $P_+$ consists of all the vertices
of the orthogonal array $C$ of weight $5$ or less.
Every subgraph of $\Qq{13}$ isomorphic to $\Qq{6}$ contains
exactly $\lambda=12$ vertices of $C$.
For every vertex $\bar u$ of weight $6$,
there is such subgraph that contains $\bar u$ and $2^6-1$ vertices
of smaller weight. Counting the number of vertices of $P_+$
among them, we can determine whether $\bar u$ belongs to $C$
or not. After finding, in this way, all weight-$6$ elements of
$C$, we can repeat the similar procedure for the weight $7$, then $8$,
$9$, $10$, $11$, $12$, and $13$.

%====================================================
%====================================================
%====================================================
\section{Results of the classification}\label{s:res}
\begin{theorem}\label{th:13}
There is only one orthogonal array OA$(1536,13,2,7)$, up to equivalence.
Its automorphism group has order $480$; 
the orthogonal array is partitioned into orbits of sizes 
$240$, $240$, $240$, $240$, $240$, $240$, $48$, $48$, 
and the complement is partitioned into $2$ orbits of size $48$,
$4$ orbits of size $80$, $18$ orbits of size $240$, and $4$ orbits of size $480$.
The kernel has size $4$ and contains words of weight $0$, $6$, $7$, and $13$.
\end{theorem}

By Proposition~\ref{p:short}, 
every OA$(768,12,2,6)$ can be obtained
by shortening some 
OA$(1536,\linebreak[3]13,\linebreak[3]2,\linebreak[3]7)$.
Since the OA$(1536,13,2,7)$ is unique up to equivalence,
shortening it in different positions we get all the OA$(768,12,2,6)$,
also up to equivalence.
Under the action of the automorphism group of the OA$(1536,13,2,7)$,
the positions are divided into three orbits by $1$, $6$, and $6$, 
corresponding
to the three equivalence classes of OA$(768,12,2,6)$.
We should also note that the OA$(1536,13,2,7)$ is invariant
under translation by $\bar 1$ 
(see Proposition~\ref{p:short} or the claim of Theorem~\ref{th:13} about the kernel); 
so, the results of $0$-shortening and $1$-shortening in the same position are equivalent.

\begin{theorem}\label{th:12}
There are three orthogonal arrays OA$(768,12,2,6)$, up to equivalence.
One of them has the automorphism group of order $240$, 
with orbit sizes $120$, $120$, $120$, $120$, $120$, $120$, $24$, $24$.
Each of the other two arrays has the automorphism group of order $40$;
two orbits of size $4$, $14$ orbits of size $20$, and $12$ orbits of size $40$.
\end{theorem}

%====================================================
%====================================================
%====================================================
\section[Representations of OA(1536,13,2,7)]{Representations of OA$(1536,13,2,7)$}\label{s:constr}
\subsection{The Fon-Der-Flaass construction}\label{s:FDF}

The following is a special case
of a construction 
% of equitable $2$-partitions of $\Qq{n}$ 
from \cite{FDF:PerfCol}. 
The construction starts 
from an equitable partition  $(C_{6},\overline C_{6})$
with quotient matrix $[[1,5],[3,3]]$.
The cell $C_{6}$
is partitioned into edges;
we use 
the notation $i(\bar c)$ 
to indicate the direction 
of the edge that contains
a vertex $\bar c$ of $C_{6}$.
To be explicit, we list all words $\bar c$ of the cell $C_6$
(which is known as OA$(24,6,2,3)$ \cite{Tarannikov2000}):
\begin{eqnarray} \label{eq:C6}
\begin{array}{r@{\ }r@{\quad}r@{\ }r@{\quad}l} 
\mathbf000000,& \mathbf100000, &   \mathbf111111,& \mathbf011111,  & i(\bar c)=1, \\
0\mathbf00110,& 0\mathbf10110, &   1\mathbf11001,& 1\mathbf01001,  & i(\bar c)=2, \\
00\mathbf0011,& 00\mathbf1011, &   11\mathbf1100,& 11\mathbf0100,  & i(\bar c)=3, \\
010\mathbf001,& 010\mathbf101, &   101\mathbf110,& 101\mathbf010,  & i(\bar c)=4, \\
0110\mathbf00,& 0110\mathbf10, &   1001\mathbf11,& 1001\mathbf01,  & i(\bar c)=5, \\
00110\mathbf0,& 00110\mathbf1, &   11001\mathbf1,& 11001\mathbf0,  & i(\bar c)=6. 
\end{array}
\end{eqnarray}
% The set $C_6$ is an OA$(24,6,2,3)$ %, found in 
% \cite{Tarannikov2000};
% %as a correlation-immune function of order $3$; 
% the partition $(C_6, \overline C_6)$
% is equitable  with quotient matrix $[[1,5],[3,3]]$.

\begin{pro}[a special case of {\cite[Proposition~2]{FDF:PerfCol}}]\label{p:FDF13}
%Define
The partition $(C_{13},\overline C_{13})$, where
\begin{equation}\label{eq:C13}
 C_{13} = \{ (\bar b|\bar b+\bar c| b_1+b_2+b_3+b_4+b_5+b_6+b_{i(\bar c)}+c_{i(\bar c)}) \, : \,  
\bar b = (b_1...b_6)\in \{0,1\}^6, \ 
\bar c= (c_1...c_6)\in C_6
\},
\end{equation} 
is equitable
with quotient matrix $[[0,13],[3,10]]$,
and $C_{13}$ is an OA$(1536,13,2,7)$.
\end{pro}

% where for every $\bar c$ from $C_6$, its only neighbor from $C_6$ differs with $\bar c$
% in the $i(\bar c)$-th coordinate (see \eqref{eq:C6}).

\begin{remark}
The Fon-Der-Flaass construction \cite{FDF:PerfCol} admits the possibility
of \emph{switching} the resulting equitable partition. 
In our  case, %the code $C_6$ is partitioned into $12$ edges. 
we can choose an edge $\{\bar c',\bar c''\}$ in $C_6$, 
and change the value of the last coordinate
for the $2^{7}$ vertices of $C_{13}$ 
corresponding in \eqref{eq:C13} to $\bar c \in \{\bar c',\bar c''\}$.
This  operation, \emph{switching}, 
results in an equitable partition with the same quotient matrix.
Since this can be done with each of the $12$ edges,
switching gives $2^{12}$ different equitable partitions.
By Theorem~\ref{th:13},
all these partitions are equivalent in the considered special case,
which can be considered as a surprising result of the classification.
In general, the construction in combination with switching
gives a huge number of inequivalent equitable partitions of $Q_n$ 
as $n$ grows \cite{VorFDF}. 
\end{remark}

%====================================================
\subsection{The Fourier transform}\label{s:Fr}

The Fourier transform of a real-valued (or complex-valued) function $f$
on $\{0,1\}^n$ is the collection of the coefficients  $\hat f(\bar y)$,
$\bar y \in \{0,1\}^{n}$,
in the expansion
$$
f(\bar x)
= \sum_{\bar y \in \{0,1\}^{n}}
\hat f(\bar y)
(-1)^{\langle \bar y, \bar x \rangle}
$$
of $f$ in terms of the orthogonal basis from the \emph{characters}
$\psi_{\bar y}(\bar x)=(-1)^{\langle \bar y, \bar x \rangle}$,
where $\langle (y_1,...,y_n),(x_1,...,x_n) \rangle=y_1x_1+\ldots+y_nx_n$.
The Fourier transform, whose variants are also known as the Walsh--Hadamard transform and 
the MacWilliams transform, 
is an important representation of a function or a set of vertices
in $\{0,1\}^n$. In particular, it is well known and straightforward that a multiset of vertices in $\{0,1\}^n$
is an OA$(N,n,2,t)$ 
if and only if the Fourier transform $\hat f$ of its multiplicity function satisfies 
$\hat f(\bar 0)=N/2^n$ and $\hat f(\bar y)=0$ for all $\bar y$ of weight $1$, $2$, \ldots, $t$.
On the other hand, a set of vertices of $\Qq{n}$  is a cell of an equitable $2$-partition of $\Qq{n}$
if and only if the nonzero ($\bar y\ne \bar 0$) nonzeros ($\hat f(\bar y)\ne 0$) 
of the Fourier transform $\hat f$ of its characteristic function
have the same weight.
The Fourier transform of $C_{13}$ was found computationally.
It can be seen from the construction in the previous subsection
that the two-cycle coordinate permutation $(2\ 3\ 4\ 5\ 6)(8\ 9\ 10\ 11\ 12)$ is an automorphism
or $C_{13}$; it follows that the Fourier transform is also invariant under this  coordinate permutation.

\begin{theorem} The Fourier decomposition
$$\chi_{C_{13}}(\bar x) = \sum_{\bar y \in \{0,1\}^{13}}
\phi(\bar y)
(-1)^{\langle \bar y, \bar x \rangle}$$
of the characteristic $\{0,1\}$-function of the orthogonal array $C_{13}$ 
defined in \eqref{eq:C13}
 has $1+111$ nonzero coefficients $\phi(\bar y)$.
 The collection of  coefficients $\phi(\bar y)$ 
 is invariant under the coordinate permutation
 $\pi=(2\ 3\ 4\ 5\ 6)(8\ 9\ 10\ 11\ 12)$. Below is the list of representatives
 $\bar y$ under $\pi$
 corresponding to the nonzero values of $\phi(\bar y)$.
 $$
 \begin{array}{r|l}
 \mbox{value of $\phi$} & \mbox{representatives under $\pi=(2\ 3\ 4\ 5\ 6)(8\ 9\ 10\ 11\ 12)$} \\ \hline
\displaystyle 3/16 & 
0\,00000|0\,00000|0
 \\[2pt]\hline
\displaystyle -1/16 & 
1\,01011|1\,01011|0
 \\[2pt]\hline
\displaystyle 1/16 & 
1\,00111|1\,00111|0,\  
0\,01111|0\,01111|0
 \\[2pt]\hline
\displaystyle -1/32 &
 0\,10100|1\,01111|1, \  
 0\,10010|1\,01111|1, \  
 0\,01111|1\,10100|1, \  
 0\,01111|1\,10010|1, \\  &
 1\,00111|0\,11100|1, \  
 1\,11100|0\,00111|1, \  
 1\,01011|0\,10101|1, \  
 1\,10101|0\,01011|1
 \\[2pt]\hline
 \displaystyle 1/32 &
1\,11111 | 1\,00000 | 1, \
1\,01110 | 1\,10001 | 1, \
1\,10101 | 1\,01010 | 1, \
1\,00100 | 1\,11011 | 1, \\  &
1\,00100 | 0\,11111 | 1, \
1\,11111 | 0\,00100 | 1, \
1\,10101 | 0\,01110 | 1, \
1\,01110 | 0\,10101 | 1, \\  &
0\,11110 | 1\,10001 | 1, \
0\,01111 | 1\,10001 | 1, \
0\,11000 | 1\,01111 | 1, \
0\,00011 | 1\,11110 | 1 \
 \end{array}
 $$
\end{theorem}
It can be noted that all $\bar y$ with $\phi = \pm 1/16$
are all the $15$ words of form $(\bar u|\bar u|0)$, where $\bar u\in\{0,1\}^6$ and $\wt(\bar u)=4$. 
Further, all $\bar y$ with $\phi = \pm 1/32$
are all the $96$ weight-$8$ words of form $(\bar u|\bar w|1)$, where $\bar u,\bar w\in\{0,1\}^6$,
$\wt(\bar u)$ is even, and the positions of zeros in $\bar u$ and $\bar w$ are disjoint.

%====================================================
%====================================================
%====================================================
\section[On OA(2048,14,2,7) and similar parameters]{On OA$(2048,14,2,7)$ and similar parameters}\label{s:14}
In this section, we construct  
orthogonal arrays OA$(2^{2^m-m-1},2^m-2,2,2^{m-1}-1)$
that cannot be extended to $1$-perfect codes of length $2^m-1$.
In particular, this
means that the characterization of the orthogonal arrays 
OA$(2048,14,2,7)$ cannot be done by characterizing 
only punctured $1$-perfect codes in $\Qq{14}$.

\begin{definition}[$1$-perfect codes and related concepts]
A set $C$ of vertices of $H(n,q)$ is called an 
\emph{$l$-fold $1$-perfect code} 
(in the case $l=1$, simply a \emph{$1$-perfect code})
if $(C,\overline C)$ is an equitable partition with quotient matrix
$$
\left(
\begin{array}{cc}
l{-}1 & n(q{-}1){-}l{+}1 \\ 
l & n(q{-}1){-}l
\end{array}
\right)
\qquad
\mbox{(in particular,} \quad 
\left(
\begin{array}{cc}
0 & n(q{-}1)\\ 
1 & n(q{-}1){-}1
\end{array}
\right) \mbox{if $l=1$)},
$$
that is, if every radius-$1$ ball in $H(n,q)$ contains exactly $l$ words of $C$.
Obviously, the union of disjoint $l$- and $l'$-fold $1$-perfect codes in $H(n,q)$
is an $(l+l')$-fold $1$-perfect code. A $2$-fold $1$-perfect code is called
\emph{splittable} (\emph{unsplittable}) if it can (respectively, cannot) be represented as the union of two $1$-perfect codes. 
A set $C$ of vertices in $\Qq{n}$ 
is called a \emph{punctured $1$-perfect code} if 
$C= C' \cup C''$ where $C'|0 \cup C''|1$ is a $1$-perfect code.
\end{definition}

 \begin{theorem}\label{th:14} For every $m\ge 4$, there is an orthogonal array 
 OA$(2^{2^m-m-1},2^m-2,2,2^{m-1}-1)$
 that is not a punctured $1$-perfect code.
 \end{theorem}
 \begin{proof}
Unsplittable $2$-fold $1$-perfect binary codes were constructed in \cite{KroPot:nonsplittable}
in every $\Qq{n}$ such that $n=2^m-1\ge 15$. 
We will construct such set  
with an additional property such that 
after shortening it gives a required orthogonal array.

At first, we need a set $M_k\subset \{0,1,2,3\}^k$, $k={2^{m-2}}\ge 4$,
of vertices of $H(k,4)$  with the following properties:
\begin{itemize}

 \item[(I)] for every word $\bar x$ in $\{0,1,2,3\}^k$ 
 and for every position $i$ from $\{1,\ldots,k\}$, 
 exactly two words from
 $M_k$ have the same values as $\bar x$ 
 in all positions may be except the $i$-th position 
 (in terms of \cite{KroPot:nonsplittable}, $M_k$ is a $2$-fold MDS code);
 
 \item[(II)] $M_k$ cannot be partitioned into two independent sets 
 (in terms of \cite{KroPot:nonsplittable}, it is unsplittable); 
 
 \item[(III)] 
 $(x_1,\ldots,x_{k-1},0)\in M_k$
 if and only if $(x_1,\ldots,x_{k-1},1)\in M_k$;
 
 $(x_1,\ldots,x_{k-1},2)\in M_k$
 if and only if $(x_1,\ldots,x_{k-1},3)\in M_k$. 
 \end{itemize}
 We construct $M_k$ in three steps.
 
 1. We start with defining $M_{2},M'_{2}\subset\{0,1,2,3\}^2$ by listing their elements:
\begin{equation}\label{eq:M0M1}
M_2=\{\underline{00},\underline{01},\underline{10},\underline{12},22,23,31,33\},\quad
M'_2=\{00,\underline{01},\underline{11},\underline{12},22,23,30,33\}.
\end{equation}

2. Define $M_3=M_2|0 \cup M'_2|1 \cup \overline M_2|2 \cup \overline M'_2|3$.

3. Recursively define $M_{i}=M_{i-1}|0 \cup M_{i-1}|1 \cup \overline M_{i-1}|2 \cup  \overline M_{i-1}|3$,
$i=4,\ldots,k$.

From step 1, we can directly check (I) for $i=1,2$. 
Step 2 guarantees (I) for $i=3$.
Step 3 guarantees (I) for $i=4,\ldots,k$ and (III).
The $7$-cycle 
induced by the vertices
$$\underline{01}0...0, \quad
\underline{00}0...0, \quad \underline{10}0...0, \quad
\underline{12}0...0, \quad \underline{12}1...0, \quad
\underline{11}1...0, \quad \underline{01}1...0 $$
supports (II) because it is impossible to distribute these $7$ elements between two independent sets.

Next, we enumerate the words of $\{0,1\}^3$:
$$ 
\begin{array}{rrrr}
\bar z_{0,0}=000, & 
\bar z_{1,0}=110, & 
\bar z_{2,0}=011, & 
\bar z_{3,0}=101, \\ 
\bar z_{0,1}=111, &
\bar z_{1,1}=001, &
\bar z_{2,1}=100, &
\bar z_{3,1}=010,
\end{array}
$$
the even-weight words of $\{0,1\}^4$: 
$$\bar y_{i,0,b}=\bar z_{i,b}|b, \qquad i=0,1,2,3,\quad b=0,1,$$
and the odd-weight words of $\{0,1\}^4$:
$$\bar y_{i,1,b}=\bar z_{i,b}|(1-b), \qquad i=0,1,2,3, \quad b=0,1.$$

We choose a $1$-perfect code in $\{0,1\}^{k-1}$ and denote it $P_{k-1}$. For example,
$P_3$ can be $\{000,111\}$.
Now, we define 
\begin{eqnarray}\label{eq:Phelps}
C_{2^m-1}&=&\big\{(\bar y_{a_{1},c_{1},b_{1}}|\ldots|\bar y_{a_{k-1},c_{k-1},b_{k-1}}|\bar z_{a_k,b_k}) \,:\, 
\\ \nonumber &&\ \ 
 (a_1,\ldots,a_k)\in M_k, \ (c_1,\ldots,c_{k-1})\in P_{k-1}, \ (b_1,\ldots,b_k)\in\{0,1\}^k
 \big\}. 
\end{eqnarray}
Construction \eqref{eq:Phelps} is a variant of the Phelps construction~\cite{Phelps84}
of $1$-perfect binary codes adopted in~\cite{KroPot:nonsplittable}
to construct $l$-fold $1$-perfect codes.
We state the following.
\begin{itemize}
 \item[(i)] $C_{2^m-1}$ is a $2$-fold $1$-perfect code \cite{KroPot:nonsplittable}
 (this property is derived from (I) and the construction of $C_{2^m-1}$).
 \item[(ii)] $C_{2^m-1}$ is unsplittable \cite{KroPot:nonsplittable}
 (from (II), we can find a sequence from odd number of codewords that cannot all belong
 to the union of two $1$-perfect codes).
 \item[(iii)] $\bar x|0 \in C_{2^m-1}$ if and only if $\bar x|1 \in C_{2^m-1}$.
 This property is new in comparing with the results of \cite{KroPot:nonsplittable}.
 It follows directly from (III) and the enumeration of $\bar z_{i,b}$:
 if we invert the last position of a codeword ending by 
 $\bar z_{i,b}$, we obtain a word ending by $\bar z_{\pi(i),\pi(b)}$,
 where $\pi=(0\,1)(2\,3)$; from (III) we see that the word obtained also belongs to $C_{2^m-1}$.
 \end{itemize}

 So, $(C_{2^m-1},\overline C_{2^m-1})$ is an equitable partition with quotient
 matrix $[[1,2^m-2],[2,2^m-3]]$. 
 By (iii), we have $C_{2^m-1}=C_{2^m-2}|0 \cup C_{2^m-2}|1$, 
 where, straightforwardly,
 $(C_{2^m-2},\overline C_{2^m-2})$ 
 is an equitable partition with quotient
 matrix $[[0,2^m-2],[1,2^m-3]]$, i.e., $C_{2^m-2}$ is an OA$(2^{2^m-m-1},2^m-2,2,2^{m-1}-1)$.
 If $C_{2^m-2}$ is a punctured $1$-perfect code,
  then $C_{2^m-1}$ includes
  a perfect code $C'|0 \cup C''|1$, where 
  $C_{2^m-2} = C' \cup C''$;
  by the definitions,  in this case
 $C_{2^m-1}$ is splittable, a contradiction.
\end{proof}

Although the punctured $1$-perfect codes in $Q_{14}$ do not exhaust
all the OA$(2048,14,2,7)$, 
yet there is some interest in classification of such codes up to equivalence.
This can be easily done by a straightforward computer-aided approach, 
using the  classification \cite{OstPot:15} of the 
$1$-perfect codes in $Q_{15}$. The calculations were simplified by utilizing
the fact that a code from the considered class is uniquely determined
by the set of its even-weight codewords. 
The second claim of the following theorem
is a side result of the classification.

\begin{theorem}\label{th:14_}
 There are exactly $14960$ equivalence classes of punctured $1$-perfect codes in $Q_{14}$.
 For codes in $14874$ of these classes, the set of even-weight codewords is equivalent to the set  
 of odd-weight codewords; for the remaining $86$ classes, this is not the case.
\end{theorem}

\begin{corollary}\label{c:14}
The number of equivalence classes of OA$(2048,14,2.7)$ is strictly greater than $14960$.
\end{corollary}

%====================================================
%====================================================
%====================================================
\section{Conclusions and open problems}
\indent\par
We classified all the orthogonal arrays
OA$(1536,13,2,7)$ and OA$(768,12,2,6)$
and proved that the classification of the orthogonal arrays
OA$(2048,14,2,7)$ cannot be completed  by
considering only the punctured $1$-perfect binary codes of length $14$.
The classification of the OA$(2048,14,2,7)$
remains an open challenging problem.
The approach developed in the current paper 
is probably too hard
to complete the case  OA$(2048,14,2,7)$, and finishing 
the classification is expected to require 
more theoretical results or(and) more computing capacity.

The existence of binary (and non-binary) orthogonal arrays attending the 
Bierbrauer--Friedman bound is not known in infinitely many cases; 
this is another challenging direction.
The Fon-Der-Flaass construction \cite[Proposition~2]{FDF:PerfCol} allows to construct equitable $2$-partitions of hypercubes for 
infinite series of quotient matrices with the first coefficient $0$. 
However, there are putative 
quotient matrices of type
$[[0,n],[c,d]]$
that are not covered by that construction.
In this direction,
the first two open questions are about 
existence of equitable partitions
with quotient matrices $[[0,25],[7,18]]$ and $[[0,27],[5,22]]$
(equivalently, orthogonal arrays OA$(7\cdot 2^{20},25,2,15)$
and OA$(5\cdot 2^{22},27,2,15)$).

Summarizing two theoretical results of the current paper, Theorem~\ref{th:3part}
and Theorem~\ref{th:14}, we can conclude
that an orthogonal array with the orthogonal-array parameters 
OA$(2^{n-m+1},n=2^m-3,2,2^{m-1}-2)$
of a
shortened punctured $1$-perfect binary code
(a code obtained by puncturing and then shortening a $1$-perfect code)
induces an equitable $3$-partition with the quotient matrix
$[[0,1,2^m-4],[1,0,2^m-4],[2,2,2^m-7]]$,
but is not necessarily a 
shortened punctured $1$-perfect code. 
This is an OA analog of similar results
for the error-correcting codes with (code) parameters of doubly or triply shortened 
$1$-perfect binary codes \cite{Kro:2m-3,Kro:2m-4,KOP:2011,OstPot:13-512-3}.
Noting the nice algebraic and combinatorial properties of equitable partitions,
it worth to look for more results showing that some classes of (optimal) combinatorial configurations are in one-to-one correspondence with classes of equitable partitions
with specially defined quotient matrices.

%====================================================
%====================================================
%====================================================
\section*{Acknowledgments}

The author is grateful
to Aleksandr Krotov for the help in programming, 
to the anonymous referee 
for valuable suggestions towards improving the paper,
and
to the Supercomputing Center of the Novosibirsk 
State University for provided computational resources.

This work was funded by the Russian Science Foundation under grant 18-11-00136.

\providecommand\path[1]{}

%    \bibliographystyle{elsarticle-num}
%    \bibliography{../../k}

\begin{thebibliography}{10}
\expandafter\ifx\csname url\endcsname\relax
  \def\url#1{\texttt{#1}}\fi
\expandafter\ifx\csname urlprefix\endcsname\relax\def\urlprefix{URL }\fi
\expandafter\ifx\csname href\endcsname\relax
  \def\href#1#2{#2} \def\path#1{#1}\fi

\bibitem{Bierbrauer:95}
J.~Bierbrauer, Bounds on orthogonal arrays and resilient functions,
  \href{http://onlinelibrary.wiley.com/journal/10.1002/(ISSN)1520-6610}{J.
  Comb. Des.} 3~(3) (1995) 179--183.
\newblock \href {http://dx.doi.org/10.1002/jcd.3180030304}
  {\path{doi:10.1002/jcd.3180030304}}

\bibitem{BGS:96}
J.~Bierbrauer, K.~Gopalakrishnan, D.~R. Stinson, Orthogonal arrays, resilient
  functions, error-correcting codes, and linear programming bounds,
  \href{http://epubs.siam.org/journal/sjdmec}{SIAM J. Discrete Math.} 9~(3)
  (1996) 424--452.
\newblock \href {http://dx.doi.org/10.1137/S0895480194270950}
  {\path{doi:10.1137/S0895480194270950}}

\bibitem{BMS:2017:few}
P.~Boyvalenkov, T.~Marinova, M.~Stoyanova, Nonexistence of a few binary
  orthogonal arrays,
  \href{http://www.sciencedirect.com/science/journal/0166218X}{Discrete Appl.
  Math.} 217~(2) (2017) 144--150.
\newblock \href {http://dx.doi.org/10.1016/j.dam.2016.07.023}
  {\path{doi:10.1016/j.dam.2016.07.023}}

\bibitem{BulRy:2018}
D.~A. Bulutoglu, K.~J. Ryan,
  \href{https://ajc.maths.uq.edu.au/pdf/70/ajc_v70_p362.pdf}{Integer
  programming for classifying orthogonal arrays},
  \href{http://ajc.maths.uq.edu.au}{Australas. J. Comb.} 7~(3) (2018) 362--385.

\bibitem{FDF:PerfCol}
D.~G. Fon-Der-Flaass, Perfect $2$-colorings of a hypercube,
  \href{http://link.springer.com/journal/11202}{Sib. Math. J.} 48~(4) (2007)
  740--745, translated from
  \href{http://www.mathnet.ru/php/journal.phtml?jrnid=smj\&option_lang=eng}{Sib.
  Mat. Zh.} 48(4) (2007), 923-930.
\newblock \href {http://dx.doi.org/10.1007/s11202-007-0075-4}
  {\path{doi:10.1007/s11202-007-0075-4}}

\bibitem{FDF:CorrImmBound}
D.~G. Fon-Der-Flaass, \href{http://mi.mathnet.ru/eng/semr149}{A bound on
  correlation immunity}, \href{http://semr.math.nsc.ru}{Sib. Ehlektron. Mat.
  Izv.} 4 (2007) 133--135.
\newblock\urlprefix\url{http://mi.mathnet.ru/eng/semr149}

\bibitem{FDF:12cube.en}
D.~G. Fon-Der-Flaass, {Perfect colorings
  of the $12$-cube that attain the bound on correlation immunity},
  \href{http://semr.math.nsc.ru}{Sib. Ehlektron. Mat. Izv.} 4 (2007) 292--295,
  in Russian. English translation: \url{https://arxiv.org/abs/1403.8091}.

\bibitem{Friedman:92}
J.~Friedman, On the bit extraction problem, in: Foundations of Computer
  Science, IEEE Annual Symposium on, IEEE Computer Society, Los Alamitos, CA,
  USA, 1992, pp. 314--319.
\newblock \href {http://dx.doi.org/10.1109/SFCS.1992.267760}
  {\path{doi:10.1109/SFCS.1992.267760}}

\bibitem{Gogsil:93:equitable}
C.~Godsil, Equitable partitions, in: D.~Mikl\'os, V.~T. S\'os, T.~Sz\"onyi
  (Eds.), Combinatorics: Paul Erd\"os is Eighty. Vol. 1, Bolyai Soc. Math.
  Stud., J\'anos Bolyai Mathematical Society, Budapest, 1993, pp. 173--192.

\bibitem{HSS:OA}
A.~S. Hedayat, N.~J.~A. Sloane, J.~Stufken, Orthogonal Arrays. Theory and
  Applications, Springer Series in Statistics, Springer, New York, NY, 1999.
\newblock \href {http://dx.doi.org/10.1007/978-1-4612-1478-6}
  {\path{doi:10.1007/978-1-4612-1478-6}}

\bibitem{KO:alg}
P.~Kaski, P.~R.~J. {\"O}sterg{\aa}rd, Classification Algorithms for Codes and
  Designs, Vol.~15 of Algorithms Comput. Math., Springer, Berlin, 2006.
\newblock \href {http://dx.doi.org/10.1007/3-540-28991-7}
  {\path{doi:10.1007/3-540-28991-7}}

\bibitem{KasPot08}
P.~Kaski, O.~Pottonen,
  \href{http://www.hiit.fi/files/admin/publications/Technical_Reports/hiit-tr-2008-1.pdf}{libexact
  user's guide, version 1.0}, Tech. Rep. 2008-1, Helsinki Institute for
  Information Technology HIIT (2008).
  
\bibitem{Khalyavin:2010.en}
A.~V. Khalyavin, Estimates of the capacity of orthogonal arrays of large
  strength,
  \href{http://link.springer.com/journal/volumesAndIssues/11970}{Mosc. Univ.
  Math. Bull.} 65~(3) (2010) 130--131.
\newblock \href {http://dx.doi.org/10.3103/S0027132210030101}
  {\path{doi:10.3103/S0027132210030101}}

\bibitem{Kirienko2002}
D.~Kirienko, On new infinite family of high order correlation immune unbalanced
  {B}oolean functions, in: Proceedings IEEE International Symposium on
  Information Theory ISIT 2002, Lausanne, Switzerland, June 30 -- July 5, 2002,
  2002, p. 465.
\newblock \href {http://dx.doi.org/10.1109/ISIT.2002.1023737}
  {\path{doi:10.1109/ISIT.2002.1023737}}

\bibitem{Kro:2m-3}
D.~S. Krotov, On the binary codes with parameters of doubly-shortened
  $1$-perfect codes, \href{http://link.springer.com/journal/10623}{Des. Codes
  Cryptography} 57~(2) (2010) 181--194.
\newblock \href {http://dx.doi.org/10.1007/s10623-009-9360-5}
  {\path{doi:10.1007/s10623-009-9360-5}}

\bibitem{Kro:2m-4}
D.~S. Krotov, On the binary codes with parameters of triply-shortened
  $1$-perfect codes, \href{http://link.springer.com/journal/10623}{Des. Codes
  Cryptography} 64~(3) (2012) 275--283.
\newblock \href {http://dx.doi.org/10.1007/s10623-011-9574-1}
  {\path{doi:10.1007/s10623-011-9574-1}}

\bibitem{KOP:2011}
D.~S. Krotov, P.~R.~J. {\"O}sterg{\aa}rd, O.~Pottonen, On optimal binary
  one-error-correcting codes of lengths $2^m-4$ and $2^m-3$,
  \href{http://ieeexplore.ieee.org/xpl/RecentIssue.jsp?punumber=18}{IEEE Trans.
  Inf. Theory} 57~(10) (2011) 6771--6779.
\newblock \href {http://dx.doi.org/10.1109/TIT.2011.2147758}
  {\path{doi:10.1109/TIT.2011.2147758}}
  
\bibitem{KroPot:nonsplittable}
D.~S. Krotov, V.~N. Potapov, On multifold {MDS} and perfect codes that are not
  splittable into onefold codes,
  \href{http://link.springer.com/journal/11122}{Probl. Inf. Transm.} 40~(1)
  (2004) 5--12, translated from
  \href{http://www.mathnet.ru/php/journal.phtml?jrnid=ppi\&option_lang=eng}{Probl.
  Peredachi Inf.} 40~(1) (2004), 6-14.
\newblock \href {http://dx.doi.org/10.1023/B:PRIT.0000024875.79605.fc}
  {\path{doi:10.1023/B:PRIT.0000024875.79605.fc}}

\bibitem{KroVor:CorIm}
D.~S. Krotov, V.~K. Vorob'ev, {On
  unbalanced {B}oolean functions attaining the bound $2n/3-1$ on the
  correlation immunity}, E-print 1812.02166v2, arXiv.org (2018).
\newblock\urlprefix\url{https://arxiv.org/abs/1812.02166}

\bibitem{MWS}
F.~J. MacWilliams, N.~J.~A. Sloane, The Theory of Error-Correcting Codes,
  Amsterdam, Netherlands: North Holland, 1977.

\bibitem{nauty2014}
B.~D. McKay, A.~Piperno, Practical graph isomorphism, {II}, J. Symb. Comput. 60
  (2014) 94--112.
\newblock \href {http://dx.doi.org/10.1016/j.jsc.2013.09.003}
  {\path{doi:10.1016/j.jsc.2013.09.003}}

\bibitem{OstPot:13-512-3}
P.~R.~J. {\"O}sterg{\aa}rd, O.~Pottonen, Two optimal one-error-correcting codes
  of length 13 that are not doubly shortened perfect codes,
  \href{http://link.springer.com/journal/10623}{Des. Codes Cryptography}
  59~(1-3) (2011) 281--285.
\newblock \href {http://dx.doi.org/10.1007/s10623-010-9450-4}
  {\path{doi:10.1007/s10623-010-9450-4}}
 
\bibitem{OstPot:15}
P.~R.~J. {\"O}sterg{\aa}rd, O.~Pottonen, The perfect binary
  one-error-correcting codes of length 15: Part {I}---classification,
  \href{http://ieeexplore.ieee.org/xpl/RecentIssue.jsp?punumber=18}{IEEE Trans.
  Inf. Theory} 55~(10) (2009) 4657--4660.
\newblock 
  \path{doi:10.1109/TIT.2009.2027525}

\bibitem{Phelps84}
K.~T. Phelps, A general product construction for error correcting codes,
  \href{http://epubs.siam.org/journal/sjamdu}{SIAM J. Algebraic Discrete
  Methods} 5~(2) (1984) 224--228.
\newblock \href {http://dx.doi.org/10.1137/0605023}
  {\path{doi:10.1137/0605023}}
  
\bibitem{Potapov:2010}
V.~N. Potapov, {On perfect colorings of
  {B}oolean $n$-cube and correlation immune functions with small density},
  \href{http://semr.math.nsc.ru}{Sib. Ehlektron. Mat. Izv.} 7 (2010) 372--382,
  in Russian, English abstract.
  
\bibitem{Pot:2012:color}
V.~N. Potapov, On perfect $2$-colorings of the $q$-ary $n$-cube,
  \href{http://www.sciencedirect.com/science/journal/0012365X}{Discrete Math.}
  312~(6) (2012) 1269--1272.
\newblock \href {http://dx.doi.org/10.1016/j.disc.2011.12.004}
  {\path{doi:10.1016/j.disc.2011.12.004}}

\bibitem{SeiZem:OA:66}
E.~Seiden, R.~Zemach, On orthogonal arrays,
  \href{http://ProjectEuclid.org/aoms}{Ann. Math. Stat.} 37~(5) (1966)
  1355--1370.
\newblock \href {http://dx.doi.org/10.1214/aoms/1177699280}
  {\path{doi:10.1214/aoms/1177699280}}

\bibitem{Tarannikov2000}
Y.~Tarannikov, \href{https://eprint.iacr.org/2000/005}{On resilient {B}oolean
  functions with maximal possible nonlinearity}, Cryptology ePrint Archive
  2000/005 (2000).
\newline\urlprefix\url{https://eprint.iacr.org/2000/005}
  
\bibitem{Taran:private2018}
Y.~V. Tarannikov, Private communication (June 2018).


  
  \bibitem{VorFDF}
K.~V. Vorobev, D.~G. Fon-Der-Flaass, \href{http://mi.mathnet.ru/eng/semr228}{On
  perfect $2$-colorings of the hypercube}, \href{http://semr.math.nsc.ru}{Sib.
  Ehlektron. Mat. Izv.} 7 (2010) 65--75, in Russian, with English abstract.


\end{thebibliography}
%    \end{document}

\providecommand\href[2]{#2} \providecommand\url[1]{\href{#1}{#1}}
  \def\DOI#1{{\small {DOI}:
  \href{http://dx.doi.org/#1}{#1}}}\def\DOIURL#1#2{{\small{DOI}:
  \href{http://dx.doi.org/#2}{#1}}}

\end{document}